\documentclass[12pt,german]{amsart}

\usepackage{hyperref}
\usepackage[dvips]{graphicx}

\newtheorem{theorem}{Theorem}[section]
\newtheorem{lemma}[theorem]{Lemma}

\theoremstyle{definition}
\newtheorem{definition}[theorem]{Definition}

\theoremstyle{remark}

\newcommand{\NN}{ {\mathbb N} }

\newcommand{\CC}{{\mathbb C}}

\newcommand{\cB}{\mathcal{B}}
\newcommand{\cA}{\mathcal{A}}

\newcommand{\kk}{\kappa}

\title[Operator-valued $S$-transform]{A short proof for the twisted multiplicativity property of the operator-valued $S$-transform}

\author{Roland Speicher}
\address{%
Universit\"at des Saarlandes,
Fachrichtung Mathematik,
Postfach 151150,
66041 Saarbr\"ucken}
\email{speicher@math.uni-sb.de}
\dedicatory{Dedicated to the memory of J\"org Eschmeier}
\thanks{The author was partially supported by the SFB-TRR 195 of the German Research Foundation DFG}

\begin{document}

\begin{abstract}
We provide a short proof for the twisted multiplicativity property of the operator-valued $S$-transform.
\end{abstract}

\maketitle

\section{Introduction}
In free probability the most basic operations are the free addititive and multiplicative convolutions, given by the sum and the product, respectively, of two freely independent random variables $x$ and $y$. Voiculescu provided with the $R$-transform \cite{V-add} and the $S$-transform \cite{V-prod} analytic functions which describe these operations via $R_{x+y}(z)=R_x(z)+R_y(z)$ and $S_{xy}(z)=S_x(z)\cdot S_y(z)$; Haagerup provided in \cite{H} different proofs of Voiculescu's results, relying on Fock space and elementary Banach algebra techniques. Those functions can also be considered as formal power series; their coefficients are then determined in terms of the moments of the considered variables and the above relations are translations into generating series of how moments of the sum or the product of free variables are determined in terms of the moments of the individual variables. In the case of the additive convolution, this is quite straightforward, the coefficients of the $R$-transform are then the free cumulants. For the multiplicative case the situation is a bit more complicated, but still one can get the above mentioned multiplicativity from the basic properties of free cumulants and playing around with formal power series; for this we refer in particular to the book \cite{NSp}, which covers the combinatorial facet of free probability. Other general introductions to the basics of free probability can be found in \cite{VDN, MSp, VSW}.

There exists a (very powerful) operator-valued extension of free probability \cite
{V1} with its operator-valued versions of the additive and multiplicative free convolutions and of the $R$-transform and the $S$-transform. Whereas for the additive case and the $R$-transform, the statements and proofs can easily be extended to the operator-valued case, the multiplicative situation is more complicated. It was actually discovered by Dykema in \cite{D1}  that in this case the formula for the $S$-transform of a product of free variables involves a twist, due to the non-commutativity of the underlying algebra of ``scalars". Though Dykema uses the language of formal power series adapted to this operator-valued setting (formal multilinear function series), both his proofs in \cite{D1} and \cite{D2} use quite involved Fock space realizations, modeled according to Haagerup's approach in the scalar-valued case \cite{H}. Our goal here is to give a more direct proof of the twisted multiplicativity of the $S$-transform, just using the basic definitions and properties of free cumulants and of the $S$-transform, as well as easy formal manipulations with power series. This is in principle just an operator-valued adaptation of the same kind of arguments from \cite{RSp}. However, since the order matters now, finding the right way of writing and manipulating the formulas was not as straighforward as it might appear from the polished final write-up. So it could be beneficial for future use by others to record those calculations here.

Reconsidering the operator-valued $S$-transform was initiated by discussions with Kurusch Ebrahimi-Fard and Nicolas Gilliers on their preprint  \cite{EFG}. There they provide an ``understanding'' of the twist from an higher operadic point of view. Our calculations here can also be seen as a more pedestrian version of parts of their work. One should note that in their setting (as well as in Dykema's combinatorial interpretation via linked non-crossing partitions) the $T$-transform, which is the multiplicative inverse of the $S$-transform, seems to be the more appropriate object. For our formal manipulations with power series there is no such distinction.

\section{Basic definitions}
We are working in an operator-valued probability space $(\cA,\cB,E)$; this means that $\cA$ is a unital algebra with a unital subalgebra $\cB\subset\cA$ and a conditional expectation $E:\cA\to \cA$, i.e., a linear map with the additional properties
$$E[b]=b,\qquad E[b_1ab_2]=b_1 E[a] b_2\qquad\text{for all $b\in\cB$, $a_1,a_2\in\cA$}.$$
Elements in $\cA$ will be called (random) variables and the basic information about them is encoded in their operator-valued moments
$$E[xb_1xb_2\cdots b_nx]\qquad (n\in\NN; b_1,\dots,b_n\in\cB)$$
or, equivalently, in their operator-valued cumulants
$$\kk(xb_1,xb_2,\dots ,xb_n,x)\qquad (n\in\NN; b_1,\dots,b_n\in\cB).$$
The latter can be defined in a recursive way via
\begin{multline}\label{eq:recursive-cumulants}
E[xb_1xb_2\cdots b_nx]=\sum_{k=0}^n 
\sum_{1\leq q_1<q_2<\dots <q_k\leq n}\\
\kk\Bigl(x E[b_1x\cdots xb_{q_1}],x E[b_{q_1+1}x
\cdots x b_{q_2}],xE[\cdots],\dots,x\Bigr) \cdot E[b_{q_k+1}x\cdots b_nx]
\end{multline}
For more precise information about operator-valued cumulants one should consult Section 9 of \cite{MSp}.
The information about those moments and cumulants will actually be stored in formal power series, which are called
multilinear function series, and we will present below a definition of operator-valued cumulants in terms of these function series.

\begin{definition}
1) A \emph{multilinear function series} $F$ is given by as sequence $F=(F_n)_{n\in\NN_0}$, where each $F_n:\cB^{\times n}\to \cB$ is an $n$-linear map on the algebra $\cB$; with the convention that $\cB^{\times 0}=\CC$ and the corresponding $F_0$ is just a constant $F_0(1)$. 
Formally we write
$$F=\sum_{n\geq 0} F_n(b_1,b_2,\dots,b_n)=F_0+F_1(b_1)+F_2(b_1,b_2)+\cdots$$

2) For such multilinear function series $F=(F_n)_{n\in\NN_0}$ and $G=(G_n)_{n\in\NN_0}$ one has the obvious operation of a sum
$$F+G:=(F_n+G_n)_{n\in\NN_0},$$
but also a formal version of a product $F\cdot G$ and, if $G_0=0$, a composition $F\circ G$.
The product is defined for all $n\geq 0$ by
$$(F\cdot G)_n(b_1,\dots,b_n)=\sum_{k=0}^n F_k(b_1,\dots,b_k) G_{n-k}(b_{k+1},\dots,b_n).$$
If $G_0=0$, the composition $F\circ G$ is defined by $(F\circ G)_0=F_0$ and for $n\geq 1$ by
\begin{multline*}
(F\circ G)_n(b_1,\dots,b_n)=
\sum_{k=1}^n \sum_{\substack{p_1,\dots,p_k\\ p_1+\cdots +p_k=n}}\\
F_k\Bigl(G_{p_1}(b_1,\dots,b_{p_1}), G_{p_2}(b_{p_1+1},\dots,b_{p_1+p_2}),\dots,
G_{p_k}(\dots,b_n)\Bigr).
\end{multline*}

3) The identity for the composition will be denoted by $I=(I_n)_{n\in\NN_0}$, and is given by
$$I_n(b_1,\dots,b_n)=\delta_{n1}b_1.$$
We will also identify the constant $1\in\cB$ with the series where only the $0$-th term is different from zero (and which is the identity for the multiplication of the series).

4) In order to avoid too many brackets we will enforce in the following the convention that composition binds stronger than product, i.e.,
$$F\cdot G\circ H\cdot K=F\cdot (G\circ H)\cdot K.$$
\end{definition}

It is quite easy to see that $F$ is invertible with respect to multiplication if and only if the constant term $F_0$ is an invertible element in $\cB$ and that $F$ is invertible with respect to composition if and only if the linear mapping $F_1:\cB\to\cB$ is invertible. For more details on this one should see \cite{D2}.

\begin{definition}
For a random variable $x\in\cA$ we define its \emph{moment series} by
$$\Phi_x=\sum_{n\geq 0} E[xb_1xb_2\cdots xb_nx]=
E[x]+E[xb_1x]+E[xb_1xb_2x]+\cdots$$ 
and its \emph{cumulant series} by
$$C_x=\sum_{n\geq 0} \kk_{n+1}(xb_1,xb_2,\dots,x b_n,x)=\kk(x)+\kk(xb_1,x)+
\kk(xb_1,xb_2,x)+\cdots$$
\end{definition}

The recursive definition \eqref{eq:recursive-cumulants} of the cumulants can then equivalently also be stated as the following relation between those two series:
\begin{equation}\label{eq:C1}
\Phi_x=C_x\circ (I+I\cdot \Phi_x\cdot I)\cdot(1+I\cdot \Phi_x)
\end{equation}
or
\begin{equation}\label{eq:C2}
\Phi_x=(1+ \Phi_x\cdot I)\cdot C_x\circ (I+I\cdot \Phi_x\cdot I)
\end{equation}

\begin{definition}
For a random variable $x\in\cA$ such that $E[x]$ is invertible the \emph{$S$-transform} $S_x$ is defined by
\begin{equation}\label{eq:S}
(1+I)\cdot\chi_x=I\cdot S_x,
\end{equation}
where $\chi_x$ is defined by
\begin{equation}\label{eq:chi}
(I\cdot\Phi_x)\circ \chi_x=I.
\end{equation}
\end{definition}

Note that the invertibility of $E[x]$ is needed in order to have an inverse $\chi_x$ of $I\cdot \Phi_x$.

Let us reformulate this in a form involving the cumulant series. This is the operator-valued version of going from formula (18.13) to formula (18.12) in \cite{NSp}

\begin{lemma}\label{lemma}
We have that
$$(I\cdot C_x)\circ (I\cdot S_x)=I\qquad\text{and}\qquad
(I\cdot S_x)\circ (I\cdot C_x)=I
$$
i.e., $I\cdot S_x$ is the inverse under composition of $I\cdot C_x$.
\end{lemma}

\begin{proof}
Since the linear term of $I\cdot C_x$ is 
$$(I\cdot C_x)_1(b)=b\kk(x)=b E[x],$$ 
and hence invertible by our assumption on $x$, we know that $I\cdot C_x$ has an inverse under composition. Thus it suffices to check the first equation in the lemma. 

By using \eqref{eq:C2}, we write $I\cdot\Phi_x$ as follows:
\begin{align*}
I\cdot \Phi_x&=(I+ I\cdot\Phi_x\cdot I)\cdot C_x\circ (I+I\cdot \Phi_x\cdot I)\\
&=(I\cdot C_x)\circ [I+I\cdot \Phi_x\cdot I]\\
&=(I\cdot C_x)\circ [(1+I\cdot \Phi_x)\cdot I]
\end{align*}
By composing this with $\chi_x$ on the right we get
$$I=(I\cdot C_x)\circ [(1+I)\cdot \chi_x]=(I\cdot C_x)\circ (I\cdot S_x)
$$
\end{proof}

\section{Twisted multiplicativity of $S$-transform}
Now we are ready to look at the $S$-transform of a product of free variables. For the definition of freeness with amalgamation as well as the definition and basic properties of the corresponding operator-valued cumulants we refer to Section 9 of the book \cite{MSp}. For our purposes the formulas \eqref{eq:Psi1}, \eqref{eq:Psi2}, \eqref{eq:Psi3}, \eqref{eq:Psi4}  below can be taken as the starting point for describing the relation between the moments of $xy$ and the moments of $x$ and of $y$. Those formulas follow quite easily from the combinatorial description of cumulants and the vanishing of mixed cumulants in free variables. We will give an indication in the proof below how this works. 

\begin{theorem}[Dykema \cite{D1}]
Let $x$ and $y$ be free with amalgamation over $\cB$. Assume that both $E[x]$ and $E[y]$ (and thus also $E[xy]=E[x]E[y]$) are invertible elements in $\cB$. Then we have
$$S_{xy}=S_y\cdot S_x\circ(S_y^{-1}\cdot I \cdot S_y).$$
\end{theorem}
One should note that the constant term of $S_y$ is given by $E[y]^{-1}$ and thus $S_y$ has a multiplicative inverse $S_y^{-1}$.
\begin{proof}
We have the  moment series of the product $xy$:
$$\Phi:=\Phi_{xy}=\sum_{n\geq 0} E[xyb_1xyb_2xyb_3\cdots xyb_nxy]=E[xy]+E[xyb_1xy]+\cdots$$ 
In addition we also define the series
$$\Phi^{y\to}=\sum_{n\geq 0}E[yb_1xyb_2xy\cdots b_n xy]=E[y]+E[yb_1xy]+E[yb_1xyb_2xy]+\cdots
$$
and
$$\Phi^{\to x}=\sum_{n\geq 0} E[xyb_1xyb_2\cdots xyb_nx]
=E[x]+E[xyb_1x]+E[xyb_1xyb_2x]+\cdots
.$$
The vanishing of mixed cumulants in free variables yields then directly the relations
\begin{equation}\label{eq:Psi1}
\Phi=C_x\circ (\Phi^{y\to}\cdot I)\cdot \Phi^{y\to}
\end{equation}
\begin{equation}\label{eq:Psi2}
\Phi=\Phi^{\to x}\cdot C_y\circ (I\cdot\Phi^{\to x})
\end{equation}
\begin{equation}\label{eq:Psi3}
\Phi^{y\to}=C_y\circ (I\cdot\Phi^{\to x})\cdot(1+I\cdot \Phi)
\end{equation}
\begin{equation}\label{eq:Psi4}
\Phi^{\to x}=(1+\Phi\cdot I)\cdot C_x\circ (\Phi^{y\to}\cdot I)
\end{equation}
To give an idea how those formulas arise, let us just write the version of \eqref{eq:recursive-cumulants} for the product $xy$, by using the fact that mixed cumulants in $x$ and $y$ vanish. Thus we only have to sum over those situations where all the arguments for the cumulant are $x$, and the $y$ are only showing up in the expectations. In order not to overload the notation, we have ignored the indices of the $b$.
\begin{multline*}
E[xybxyb\cdots bxy]=\sum_{\text{all possibilities}}\\
\kk\Bigl(x E[ybxy\cdots xyb],x E[ybxy
\cdots xy b],\dots,xE[y\cdots xy b],x\Bigr) \cdot E[ybxy\cdots bxy]
\end{multline*}
This formula results then directly in formula \eqref{eq:Psi1} for the corresponding function series. The three other formulas arise in similar ways; note that in \eqref{eq:Psi2} we are using an expansion in cumulants of $y$, and for this the version of \eqref{eq:recursive-cumulants} is more appropriate where in the sum one is not fixing the first $y$ as element in the cumulant, but the last one.  

The rest is now just playing around with those relations in order to rewrite them in terms of the $S$-transforms. 

From \eqref{eq:Psi2} we get
$$
I\cdot\Phi=I\cdot\Phi^{\to x}\cdot C_y\circ (I\cdot\Phi^{\to x})=(I\cdot C_y)\circ (I\cdot\Phi^{\to x})
$$
and thus, with $\chi:=\chi_{xy}$,
$$
I=(I\cdot\Phi)\circ \chi=(I\cdot C_y)\circ (I\cdot\Phi^{\to x})\circ \chi.
$$
Since the inverse under composition of $I\cdot C_y$ is given by $I\cdot S_y$ we see thus that we must have
\begin{equation}\label{eq:IPsi}
(I\cdot \Phi^{\to x})\circ \chi=I\cdot S_y.
\end{equation}

From \eqref{eq:Psi3} and \eqref{eq:Psi2} we get
\begin{equation*}
I\cdot\Phi^{\to x}\cdot\Phi^{y\to}=I\cdot\Phi^{\to x}\cdot C_y\circ (I\cdot\Phi^{\to x})\cdot(1+I\cdot \Phi)
=I\cdot\Phi\cdot (1+I\cdot \Phi)
\end{equation*}
Composing with $\chi$ and using \eqref{eq:IPsi} leads to
$$I\cdot S_y\cdot \Phi^{y\to}\circ \chi=I\cdot (1+I)$$
and thus to
$$S_y\cdot \Phi^{y\to}\circ \chi= (1+I),$$
hence
\begin{equation}\label{eq:Phichi}
\Phi^{y\to}\circ \chi=S_y^{-1}\cdot (1+I).
\end{equation}

Now we use \eqref{eq:Psi1} to calculate
\begin{equation*}
\Phi^{y\to}\cdot I\cdot\Phi=\Phi^{y\to}\cdot I \cdot C_x\circ (\Phi^{y\to}\cdot I)\cdot \Phi^{y\to}=(I\cdot C_x)\circ (\Phi^{y\to}\cdot I)\cdot \Phi^{y\to}
\end{equation*}
and thus, by first multiplying with $(\Phi^{y\to})^{-1}$ from the right and then using Lemma \ref{lemma},
\begin{equation*}
(I\cdot S_x)\circ(\Phi^{y\to}\cdot I\cdot\Phi\cdot (\Phi^{y\to})^{-1})= \Phi^{y\to}\cdot I.
\end{equation*}
Composing again with $\chi=\chi_{xy}$ and using \eqref{eq:Phichi} both for $\Phi^{y\to}$ and its inverse leads now to
$$(I\cdot S_x)\circ(S_y^{-1} \cdot (1+I)\cdot I \cdot (1+I)^{-1}\cdot S_y)=S_y^{-1}\cdot (1+I)\cdot\chi_{xy}
=S_y^{-1}\cdot I\cdot S_{xy}$$
or, by also using that $(1+I)\cdot I \cdot (1+I)^{-1}=I$,
$$S_y^{-1}\cdot I  \cdot S_y\cdot S_x\circ(S_y^{-1}\cdot I \cdot S_y)=S_y^{-1}\cdot I\cdot S_{xy}$$
and thus finally
$$S_y\cdot S_x\circ(S_y^{-1}\cdot I \cdot S_y)=S_{xy}.$$
\end{proof}

\section*{Data Availability}
Data sharing not applicable to this article as no datasets were generated or analysed during the current study.


\begin{thebibliography}{99}

\bibitem{D1} K. Dykema: On the S-transform over a Banach algebra. \textit{Journal of Functional Analysis} 231(1), 2006, 90--110.

\bibitem{D2} K. Dykema: Multilinear function series and transforms in free probability theory. \textit{Advances in Mathematics} 208(1), 2007, 351--407.

\bibitem{EFG} K. Ebrahimi-Fard and N. Gilliers: On the twisted factorization of the $T$-transform. Preprint 2021, arXiv:2105.09639

\bibitem{H}
U. Haagerup: On Voiculescu's R-and S-transforms for free non-commutative random variables. \textit{Free Probability theory}, Fields Inst. Commun. 12, 1997, 127--148.

\bibitem{MSp} J. Mingo and R. Speicher: \textit{Free probability and random matrices}. New York: Springer, 2017

\bibitem{RSp} N.R. Rao and R. Speicher: Multiplication of free random variables and the S-transform: The case of vanishing mean. \textit{Electronic Communications in Probability} 12, 2007, 248--258.

\bibitem{NSp} A. Nica and R. Speicher: \textit{Lectures on the combinatorics of free probability}. Cambridge University Press, 2006

\bibitem{V-add}
D. Voiculescu: Addition of certain non-commuting random variables. \textit{Journal of Functional Analysis} 66.3, 1986, 323--346.

\bibitem{V-prod}
D. Voiculescu: Multiplication of certain non-commuting random variables. \textit{Journal of Operator Theory} (1987): 223-235.

\bibitem{V1}
D. Voiculescu: Operations on certain non-commutative operator-valued random variables. \textit{Astérisque} 232.1, 1995, 243--275.

\bibitem{VDN}
D.V. Voiculescu, K. Dykema, and A. Nica: \textit{Free random variables}. American Mathematical Soc., 1992.

\bibitem{VSW}
D.V. Voiculescu, N. Stammeier, and M. Weber: \textit{Free probability and operator algebras}. European Mathematical Society, 2016.


\end{thebibliography}
\end{document}